\documentclass{article}
\usepackage{latexsym,amssymb,amsfonts,amsmath,amsthm,graphicx,enumerate}
\usepackage{dsfont}
\usepackage[latin1]{inputenc}
\usepackage{hyperref}
\numberwithin{equation}{section}
\newtheorem{thm}{Theorem}
\newtheorem{lemma}[thm]{Lemma}
\newtheorem{prop}[thm]{Proposition}
\newtheorem{cor}[thm]{Corollary}
\newtheorem{defi}[thm]{Definition}
\newtheorem{rem}[thm]{Remark}

\DeclareMathOperator{\E}{\mathsf{E}}
\DeclareMathOperator{\V}{\mathsf{V}}
\DeclareMathOperator{\G}{\mathsf{G}}

\DeclareMathOperator{\id}{{\rm id}}

\DeclareMathOperator{\Sfin}{S_{\rm fin}}
\DeclareMathOperator{\Vfin}{\mathsf{V}_{\rm fin}}
\DeclareMathOperator{\Gfin}{\mathsf{G}_{\rm fin}}
\DeclareMathOperator{\Efin}{\mathsf{E}_{\rm fin}}

\title{The $L^2$-strong maximum principle on arbitrary countable networks}
\author{Stefano Cardanobile\footnote{cardanobile@bccn.uni-freiburg.de, 
Bernstein Center for Computational Neuroscience,  
Hansastra{\ss}e 9A, D-79081 Freiburg, Germany
}
}

\begin{document}
\maketitle
\begin{abstract}
We study the strong maximum principle for the heat equation associated with the Dirichlet form 
on countable networks.
We start by analysing the boundedness properties of the incidence operators
on a countable network.
Subsequently, we prove that the strong maximum principle is equivalent to the underlying graph being
connected after deletion of the nodes with infinite degree.
Using this result, we prove that the number of connected components of the graph 
with respect to the heat flow equals the number of maximal invariant ideals of the
adjacency matrix.
\end{abstract}

\section{Introduction}
The study of the heat equation on networks has a long tradition 
both in the physical and mathematical~\cite{Pau36,Lum80}
literature.
Beside more concrete motivations, these investigations are of interest
in order to understand which properties of the heat equation 
on domains of $\mathbb R^d$ also hold (or fail to) in more general situations.

Although networks are simple, one-dimensional structures, 
it turns out that interesting phenomena already arise with respect to this kind of problems.
As an example, it has been proved by different authors  
that there exists non isomorphic graphs such that the Laplace operators on the corresponding networks
are isospectral \cite{Bel01,BanShaSmi06}.
This solves Kac's conjecture~\cite{Kac66} in the case of networks.

A further well-known property of the heat equation on a domain is the strong maximum principle:
if a positive initial data $u(0,\cdot)$ is localized on a subdomain $\omega \subset \Omega$, 
i.e.\ $u(0,x)=0$ for almost every $x\not\in\omega$, 
the resulting distribution is strictly positive for all $x \in \Omega$ and all $t>0$. 
In semigroup theory, this property is known under the name of \emph{irreducibility} and, if $(e^{t\Delta})_{t\geq 0}$ is the semigroup solving the heat equation in the sense of~\cite{Kat66}, then it is said to be \emph{irreducible}.

A possible approach for the analysis of the heat equation of the network is the variational one.
In this approach, a suitable Hilbert space is defined, and the Laplace operator is defined as
the operator associated with the Dirichlet form in this Hilbert space. 
Subsequently, the heat equation is solved weakly and classical solutions are obtained by regularity results.

We prove that irreducibility fails to hold for the heat equation in a Hilbert space context,
if nodes with infinite degree are present.
Further, we characterize those networks with nodes with infinite degrees for which the strong maximum
principle holds.
The heat equation on locally finite networks has been studied
by many authors, both in the $L^2$-setting~\cite{Cat99} and in the $L^\infty$-one~\cite{BelLub04}.
The maximum principle for semilinear parabolic network equations has been studied
in~\cite{Bel91}.
Nevertheless, literature on heat equations on networks that are not locally finite are relatively sparse
\cite{McG91,Car08}.
We also mention that irreducibility for topologically connected networks
in the finite case has been proved in~\cite{KraMugSik07},
and, as a matter of fact, our techniques are an extension to the infinite case 
of the techniques developed there.

The paper is organized as follows. 
In Section~\ref{sec:general} we set up a framework that allows us to deal with infinite networks with infinite degree.
In particular, we prove several properties of the incidence matrix of an infinite graph
that are needed in the definition of the domain of the Dirichlet form.

In Section~\ref{sec:infinitesymmetries} we discuss the strong maximum principle 
of the heat equation on a infinite network, 
proving that it possibly fails to hold for networks with infinite degree.
We finally show that how the notion of connectedness arising from the maximum principle
correctly generalizes the theorem relating the multiplicity of the eigenvalue 0 of the Laplace operator of a graph
and the number of connected component of the graph. 

We finally mention that the results 
are partially adapted and refined from~\cite{Car08}.


\section{General results}\label{sec:general}
Intuitively, a countable, oriented graph consists of vertexes $\mathsf v \in \mathsf V$ 
and oriented edges $\mathsf e \in \mathsf E$ that connect two different vertices.
The relations between the vertexes and edges 
are specified by a mapping $\partial: \mathsf E \to \mathsf V \times \mathsf V$ 
encoding the start and the end of each edge.
In fact, an \emph{oriented graph} is any triple $(\mathsf V,\mathsf E,\partial)$, where $\mathsf V, \mathsf E$ are sets 
and $\partial: \mathsf E \to \mathsf V \times \mathsf V$ is a mapping.

We recall some basic definitions.\
The \emph{degree} of a node $\mathsf v$ is the number of edges $\mathsf e$ such that $\mathsf v \in \partial \mathsf e$. 
The outbound and inbound degree are defined in an analogous manner.
Moreover, we define $\Gamma^+(\mathsf{v})$ the set of edges ending at $\mathsf{v}$ and $\Gamma^-(\mathsf v)$ the set of edges starting at $\mathsf v$.
The degree of $\mathsf v$ satisfies $\mathrm{deg}(\mathsf{v}) = |\Gamma^+(\mathsf{v})| + |\Gamma^-(\mathsf{v})|$ 
and the \emph{outbound star} centered at $\mathsf{v}$ is defined as the triple
$(\{\mathsf{v}\} \cup \partial_2(\Gamma^-(\mathsf{v})), \Gamma^-(\mathsf{v}), \partial_{| \Gamma^-(\mathsf{v})})$,
and it is, in fact, the subgraph induced by the edges outgoing from the vertex $\mathsf{v}$. 
The \emph{inbound star} is defined analogously, as well as the \emph{star} centred at $\mathsf v$.
Let us formulate explicitly the definition of the incidence matrices.
\begin{defi}\label{incidence}
The \emph{incoming incidence matrix} $\mathcal I^+$  is defined by
\begin{equation}\label{incinc}
\iota^+_{\mathsf{ v e}}:=\begin{cases}
1 & \mbox{if the edge $\mathsf e$ ends in the node $\mathsf v$,}\\
0 & \mbox{otherwise.}
\end{cases}
\end{equation}
The \emph{outgoing incidence matrix} $\mathcal I^-$ is defined by
\begin{equation}\label{outinc}
\iota^-_{\mathsf{ve}}:=
\begin{cases}
1 & \mbox{if the edge $\mathsf e$ starts in the node $\mathsf v$,}\\
0 & \mbox{otherwise.}
\end{cases}
\end{equation}
The \emph{incidence matrix} of the graph $\G$ is the matrix $\mathcal I=\mathcal I^+-\mathcal I^-.$
\end{defi}

We now fix a graph $\mathsf G$ and associate to each edge a copy of the interval $[0,1]$, such that, 
defining the Hilbert space
$$
L^2(\G):=\bigoplus_{\mathsf e \in \mathsf E} L^2(0,1),
$$
we provide the graph $\mathsf G$ with a measure-theoretic structure.
We consequently call $L^2(\mathsf G)$ an \emph{oriented network}.
For functions $\psi \in L^2(\mathsf G)$ we may and do write $\psi =: (\psi_{\mathsf e})_{\mathsf e \in \mathsf E}$.

\begin{rem}[Assumption on countable graphs]
Our goal is to prove some properties of the heat equation on $L^2(\mathsf G)$.
If $\mathsf G$ is not countable, then the space $L^2(\mathsf G)$ is not separable, 
and so it is possible to decompose the space in separable ideals that are invariant under the action of the heat equation.
Each of them corresponds to a countable subset of the edges set.
So, there is no loss of generality in considering only countable graphs and we assume this in the following.
\end{rem}
\begin{rem}[Assumption on trivial ideals]
We recall that an ideal of the Hilbert lattice $L^2(\mathsf G)$ is a subspace of the form $L^2(\omega)$,
where $\omega$ is a measurable set. To avoid trivial cases we always assume that $|\omega|>0$.
\end{rem}

Consider the space
$$
V_0:= \bigoplus_{\mathsf e \in \mathsf E} H^1(0,1).
$$
As a consequence of the boundedness of the trace operator on $H^1(0,1)$, both $\psi(0)$ and $\psi(1)$ are in $\ell^2(\mathsf E)$, 
and so the incidence matrix is a (possibly unbounded) operator from $\ell^2(\mathsf E)$ to $\ell^2(\mathsf V)$.
If we now define $V\subset V_0$ by
\begin{equation}\label{formdomain}
V:=\left\{ \psi \in V_0 :\exists d^\psi \in \ell^2(\V): \begin{array}{l}(\mathcal I^+)^\top d^\psi=\psi(0)\\(\mathcal I^-)^\top d^\psi=\psi(1)\end{array} \right\},
\end{equation}
then Definition~\ref{incidence} implies that all functions in $V$ are \emph{continuous on the graph}, in the sense that each $\psi_{\mathsf e}$ is continuous and if, e.g.\ $\mathsf e$  ends and $\mathsf e'$ starts in $\mathsf v$, it follows that $\psi_{\mathsf e}(1)=\psi_{\mathsf e'}(0)$.

We define the Laplace operator on a network as the operator associated with 
the Dirichlet form defined on the space $V$. 
To do this, we need to prove that $V$ is an Hilbert space and that $V$ is densely defined and continuously embedded in the space $L^2(\mathsf G)$.
Since we want to incorporate the possibility of nodes with unbounded degree, we have to clarify in an operator theoretic sense the expressions involving the incidence matrices in Equation~\eqref{formdomain}. This is the goal of the present section.

In the Equation~\eqref{formdomain} the existence of a $d^\psi$ 
with the required properties has to be understood as the existence of $d^\psi$ in the domain of the transpose of the incidence matrix, interpreted
as a operator from $\ell^2(\mathsf V)$ to $\ell^2(\mathsf E)$.
Before we turn our attention to these domains, we fix some notations.
Assume that $f:A \to H$ is a function from a set $A$ to a vector space $H$. If and $B\subset A$ is a subset, we define
$\pi_B f$ by
$$
\pi_{B}f(a)=
\begin{cases}
f(a), & a \in B,\\
0, & \mbox{otherwise.}
\end{cases}
$$
Observe that if $A$ is a measure space, $B$ is a measurable subset of $A$, and $H$ is a Hilbert space, 
then $P_B:f \mapsto \pi_B f$ is the orthogonal projection of $L^2(A)$ onto the ideal $L^2(B)$.

We start by proving that the incidence operators are densely defined. 
As a consequence, the transpose can be identified with the adjoint.

\begin{prop}\label{denseinc}
Both $\mathcal I^+$ and $\mathcal I^-$ have dense domain as operators from $\ell^2(\E)$ to $\ell^2(\V)$ for every countable graph $\mathsf G$.
\end{prop}

\begin{proof}
The idea of the proof is the following: we prove the claim for locally finite graphs and for infinite stars; 
we conclude combining both results.

Assume that the graph $\mathsf G$ is locally finite, i.e.\ that each node has finite degree. 
Then,
$$
\G= \bigcup_{n \in \mathbb N} \G_n,
$$
where $\G_n$ is the subgraph induced by the subset $\V_n$ of nodes having degree less than $n$.

For all $y \in \ell^2(\E)$
$$
\lim_{n \to \infty} \| y- \pi_{\E_n}(y)\|_{\ell^2(\E)}=0.
$$
We denote by $\E_n$ the set of the edges belonging to $\G_n$. The estimate
\begin{eqnarray*}
\|\mathcal I^+ \pi_{\E_n}y\|^2_{\ell^2(\V)}& = &\sum_{\mathsf{v} \in \V_n }|  \sum_{\mathsf{e} \in \Gamma^+(\mathsf{v})} \pi_{\E_n}y_{\mathsf{e}} |^2 \\
&\leq & M_n \sum_{\mathsf{v} \in \V_n}  \sum_{\mathsf{e} \in \Gamma^+(\mathsf{v})} |\pi_{\E_n}y_{\mathsf{e}} |^2  \\
& = & M_n\|\pi_{\E_n}y\|^2_{\ell^2(\E)}
\end{eqnarray*}
yields that $\pi_{\E_n}y\in D(\mathcal I^+)$ for all $n$, and so the claim is proved for a locally finite graph.

\smallskip
If the graph consists of a single inbound infinite star $S$, then $\ell^1(\E) \subset D(\mathcal I^+)$, 
and so $D(\mathcal I^+)$ is dense in $\ell^2(\E)$.

\smallskip
To complete the proof, we assume without loss of generality 
that the nodes with infinite out-degree are labeled $\mathsf{v}_k$, $k \in \mathbb N$, 
and we decompose the graph as
\begin{equation}\label{hilfdenseinc1}
\G=\bigcup_{k \in \mathbb N} \Gamma(\mathsf v_k) \cup \Gfin.
\end{equation}
Here $\Gfin$ represents the subgraph induced by the node with finite degree, 
and $\Vfin$, $\Efin$ are the corresponding vertex and edges subsets, respectively.
We define for all $n \in \mathbb N$ the approximations 
$$
\G_n = \bigcup_{k \leq n-1} \Gamma(\mathsf{v}_k) \cup \Gfin.
$$
We fix an arbitrary $x \in \ell^2(\E)$ and define 
$$
v_0:= \pi_{\Efin}x , \quad v_k:=\pi_{\Gamma(\mathsf v_k)}x,\qquad k \in \mathbb N,
$$
where we have identified $\Gamma(\mathsf v_k)$ with its edge set.
Since $\Gfin$ is locally finite, there exists a sequence 
$(v^n_0)_{n \in \mathbb N} \in D(\mathcal I^+_{|\Efin})$
such that the estimate
$$\|v^n_0 - v_0\| \leq \frac{1}{2^{n}}, \qquad n \in \mathbb N$$
holds.
In particular, extending $v_0^n$ by 0 yields a sequence in $D(\mathcal I^+)$, since $\mathcal I^+ \ell^2(\Efin) \subset \ell^2(\Vfin$).
We recall that the domain of the incidence operators is dense for infinite stars, too. 
So, for all $k \geq 1$  there exists a sequence $(v^n_k)_{n \in \mathbb N} \in D(\mathcal I^+_{|\Gamma)\mathsf v_k)})$ such that 
$$\|v^n_k - v_k\| \leq \frac{1}{2^{n+k}}, \qquad k\geq 1, n \in \mathbb N.$$
Again by the same arguments as for finite part, extending $v_k$ by 0 yields a vector $D(\mathcal I^+)$.
With an abuse of notation, we denote the extensions of $v_k^n$ $k,n\in \mathbb N$ also by $v_k^n$.

\smallskip
Summing up, for all $k \in \mathbb N$ there is a sequence $(v^n_k)_{n \in \mathbb N} \in D(\mathcal I^+)$ such that 
$$\|v^n_k - v_k\| \leq \frac{1}{2^{n+k}}, \qquad k,n \in \mathbb N.$$
We define $x^n:=\sum_{k \leq n} v_k$ and fix $\varepsilon <0$.
Since~\eqref{hilfdenseinc1} holds, there exists $n_0 \in \mathbb N$, $\|x-\pi_{\mathsf E_n}x\| < {\varepsilon}$ for all $n \geq n_0$. 
For such an $n$ we estimate
\begin{eqnarray*}
\|x-x^n\| & = & \|x-\pi_{\mathsf E_n} (x) + \pi_{\mathsf E_n}(x) - x^n\| \\
& \leq&  \|x-\pi_{\mathsf E_n} (x)\|+\|\pi_{\mathsf E_n}(x) - x^n\| \\
& <& {\varepsilon} + \frac{1}{2^n},
\end{eqnarray*}
and so $\lim_{n \to \infty }x^n = x$. 

\smallskip
Since now $x^n$ is a finite linear combination of elements in the domain,
we have that $x^n \in D(\mathcal I^+)$, thus concluding the proof.
\end{proof}
\begin{rem}[Domain of the adjoint]
One could ask whether the adjoints of the incidence operators $(\mathcal I^+)^\top,(\mathcal I^-)^\top$ 
are themselves densely defined.
In fact, this is the case if and only if the graph $\G$ is locally finite.
To see this, assume first that $\G$ is locally finite and fix a vector $x \in c_{00}(\V)$. 
Then $(\mathcal I^+)^\top \in c_{00}(\E) \subset \ell^2(\E)$, too.
This implies $x \in D((\mathcal I^+)^\top)$ and so the operator is densely defined. 
Conversely, if $(\mathcal I^+)^\top $ is densely defined, then $\mathsf 1_{\mathsf{v}}$ has to be in the domain for all $\mathsf{v} \in \V$. 
Observe that $(\mathcal I^+)^\top \left( \mathsf 1_{\mathsf{v}} \right)=\mathsf 1_{\Gamma^+(\mathsf{v})}$ which is in $\ell^2$ if any only if $\Gamma^+(\mathsf{v})$ is finite.
As a side remark, observe that this implies that $V$ is not dense in $\oplus_{\mathsf{e} \in \E} H^1(0,1)$ in the $H^1$-norm. 
However, it is not difficult to prove that both $V$ and $\oplus_{\mathsf{e} \in \E} H^1(0,1)$ are dense in $\oplus_{\mathsf{e} \in  \E} L^2(0,1)$ with respect to the $L^2$-norm.
\end{rem}
The issue whether an infinite matrix defines a (bounded) operator in a Hilbert space is known 
at least since Halmos~\cite{Hal82} to have no ``{elegant and useful answer}''.
In the context of graphs, Mohar~\cite{Moh82} has investigated the boundedness of 
the adjacency matrix, proving that boundedness is equivalent to the graph being 
uniformly locally finite (in short \emph{ULF}). 
In the next result we investigate the boundedness of the incidence matrices.
\begin{prop}\label{incidencebound}
Consider a countable graph $\mathsf G$. Then:
\begin{enumerate}[a)]
\item The incidence matrices $\mathcal I^+, \mathcal I^-$ are bounded operators from $\ell^2(\E)$ to $\ell^2(\V)$ if and only if the graph $\G$ is uniformly locally finite.
\item The incidence matrices $\mathcal I^+, \mathcal I^-$ are bounded operators from $\ell^\infty(\E)$ to $\ell^\infty(\V)$ if and only if the graph $\G$ is uniformly locally finite.
\item The operators $\mathcal I^+, \mathcal I^- $ are contractive from $\ell^1(\E)$ to $\ell^\infty(\V)$.
\end{enumerate}
\end{prop}

\begin{proof}
We start proving \emph{a)}. We assume that the graph $\G$ is \emph{ULF} with 
maximal degree $D$, 
fix $x \in \ell^2(\E)$ and compute, again using the same symbol for a star and its edges set
\begin{eqnarray*}
\|\mathcal I^+x\|^2_{\ell^2(\V)}&=& \sum_{\mathsf{v} \in \V}  |\sum_{\mathsf{e} \in \Gamma^+(\mathsf{v})}  x_{\mathsf{e}}|^2\\
&\leq& \sum_{\mathsf{v} \in \V}  \|(x_{\mathsf{e}})_{\mathsf{e} \in \Gamma^+(\mathsf{v})}\|_{\ell^1(\Gamma^+(\mathsf{v}))}^2\\
&\leq& \sum_{\mathsf{v} \in \V} \deg^+(\mathsf{v}) \|(x_{\mathsf{e}})_{\mathsf{e} \in \Gamma^+(\mathsf{v})}\|_{\ell^2(\Gamma^+(\mathsf{v}))}^2\\
&\leq& D\|x\|^2_{\ell^2(\E)}.
\end{eqnarray*}
Alternatively, if the graph is locally finite, but it is not \emph{ULF}, 
then there exists a sequence of nodes $(\mathsf{v}_\ell)_{\ell \in \mathbb N}$ such that $\lim_{\ell \to \infty}\deg(\mathsf{v}_\ell)=\infty$. 
Consider the vectors $x_{\ell}:= \frac{1}{\sqrt{\deg^+(\mathsf{v}_\ell)}}\mathsf 1_{I(\mathsf{v}_\ell)}.$ Then $\|x_{\ell}\|_{\ell^2(\E)}=1$, but $\|\mathcal I^+x_{\ell}\|_{\ell^2(\V)}=\deg^+(\mathsf{v}_\ell).$ 
This shows that $\mathcal I^+$ is not bounded.

If, finally, there exists a node $\mathsf{v}$ such that $\deg^+(\mathsf{v})=\infty$,
it suffices to observe that for the inbound star $\Gamma^+(\mathsf{v})$
all vectors $0 \leq x \in \ell^2(\Gamma^+(\mathsf{v}))\setminus \ell^1(\Gamma^+(\mathsf{v}))$ are not in the domain of $\mathcal I^+_{|\Gamma^+(\mathsf{v})}.$
Extending one of these vectors by 0 yields the claim.

\smallskip
To see that \emph{b)} holds we observe that the operator $\mathcal I^+$ is a positive matrix. 
Thus, it is sufficient to compute $\mathcal I^+ \mathsf 1_{\mathsf{e}}
= (\deg^+ (\mathsf{v}))_{\mathsf{v} \in \V}$.

\smallskip
Again, since $\mathcal I^+$ is a positive matrix to see that \emph{c)} holds, we compute for arbitrary $x \in \ell^1$
$$
\|\mathcal I^+x\|_{\ell^\infty(\mathsf V)}
\leq \|\mathcal I^+x\|_{\ell^1(\mathsf V)} 
= \|x\|_{\ell^1(\mathsf E)}.
$$
\end{proof}
In this section we have established some fundamental properties of the incidence matrices appearing in the definition of the space $V$.
In the next section, we will prove some irreducibility properties for the Laplace operator on $L^2(\mathsf G)$.

\section{Irreducibility for the heat semigroup}\label{sec:infinitesymmetries}
We introduce the \emph{Laplace operator on a network} as the operator associated with the symmetric, bilinear form
$$
a(f,g):=\int_\mathsf G f'(x) \overline{g'(x)} dx
$$
with form domain $D(a):=V$ . 
We call the semigroup $(e^{t\Delta})_{t\geq 0}$ on the Hilbert space $L^2(\mathsf G)$ 
generated by the Laplace operator the \emph{heat semigroup}.
Due to a Theorem of Ouhabaz~\cite{Ouh05} it is possible to characterize the invariance
under the action of $(e^{t\Delta})_{t\geq 0}$
of closed, convex subsets $C$ of $L^2(\mathsf G)$, i.e.\ the property
$$
f\in C \Rightarrow [\forall t\geq 0 : e^{t\Delta}f \in C].
$$
The aforementioned Theorem takes a particularly simple 
form if $C$ is a linear subspace.
Since we will use this simplified version several times, we formulate it explicitly
for the sake of readability.

\begin{thm}[Invariance of linear subspaces]\label{ortho}
Consider a densely defined, continuous, elliptic sesquilinear form $(a,V)$ on the Hilbert space $H$ and fix a closed linear subspace $Y \subset H$.
Then $Y$ is invariant under the semigroup generated by $a$ if and only if $P_Y V \subset V$ and 
$$
a(f,g)=0, \qquad f \in V\cap Y, g \in V \cap Y^\perp.
$$
\end{thm}

In particular, the theorem can be used to characterize irreducibility, if the underlying Hilbert space is of the form $L^2(\Omega)$.
We recall that a semigroup is irreducible if and only if whenever $L^2(\omega)$ is invariant under the action of the semigroup,
it follows that either $|\Omega \setminus \omega|=0$ or $|\omega|=0$.
In the context of graphs, Theorem~\ref{ortho} implies that irreducibility is equivalent to the invariance of continuity under the projection on $L^2(\omega)$.
This observation has been used in~\cite{KraMugSik07} to prove that irreducibility is equivalent to the graph being connected, in the case of a finite graph.
The same arguments of~\cite{KraMugSik07} fail to hold for infinite graphs, and, 
as a matter of fact, the equivalence does not hold, as proved in~\cite{Car08}.

The key observation is that on nodes with infinite degree Dirichlet boundary conditions are automatically imposed, 
and so an initial data localised on a side of such nodes cannot propagate to the other side. 
Thus, irreducibility in connected, infinite graphs is equivalent to the graph being connected after deletion of nodes with infinite degree.
We start proving the result concerning the boundary conditions imposed nodes with infinite degree.

\begin{lemma}\label{hilflemma3}
For all countable graphs with vertex set $\mathsf V$ the and all $\mathsf{v} \in \V$ the following assertions are equivalent.
\begin{enumerate}[a)]
\item $\deg(\mathsf{v})<\infty$
\item $\exists \psi \in V: \pi_{\mathsf{v}}(d^\psi)\neq 0$
\end{enumerate}
\end{lemma}
\begin{proof}
Recall that $H^1(0,1) \hookrightarrow C[0,1]$ and so $\|\psi\|_{H^1(0,1)} \geq M \|\psi\|_\infty$ and
fix an arbitrary node $\mathsf{v} \in \V$.

We first prove $b) \Rightarrow a)$. If there exists $\psi \in V$, $\psi(\mathsf{v})\neq0$, then
$$
\|\psi\|^2_V = \sum_{\mathsf{e} \in \E} \|\psi_{\mathsf{e}}\|^2_{V_{\mathsf{e}}} \geq  M \sum_{\mathsf{e} \in \Gamma(\mathsf{v})} \|\psi_{\mathsf{e}}\|^2_\infty \geq M \sum_{\mathsf{e} \in \Gamma(\mathsf{v})} |\psi(\mathsf{v})|^2,
$$
and so $\Gamma(\mathsf{v})$ has to be finite.

Conversely, choose $\mathbb C \ni \lambda \neq 0$ and set
$$
\psi(\mathsf{v}')=
\begin{cases}
\lambda, & \mathsf{v}' =\mathsf{v},\\
0, & \mbox{otherwise}.
\end{cases}
$$
For all $x \in \G \setminus \V$ interpolate $\psi$ by affine functions. Then $\psi_{\mathsf{e}} \in H^1(0,1)$ for all $\mathsf{e} \in \E$ and moreover
$$
\|\psi\|^2_{L^2}= \deg(\mathsf{v}) \frac{|\lambda^2|}{3},\quad \|\psi\|^2_{H^1}= \deg(\mathsf{v}) \frac{4 |\lambda|^2}{3}.
$$
Finally, $\psi$ is continuous in the nodes. So, $\psi \in V$ and this completes the proof.
\end{proof}

For a subset of nodes $\V'$ we call the \emph{subgraph induced by $\V'$} the subgraph of $\G$ containing all edges that are only incident to nodes of $\V'$.
The \emph{boundary} of $\G'$ consists of the nodes of $\G'$ that are adjacent to nodes of $\G\setminus \G'$.
\begin{prop}\label{infiniteinvariance}
Consider the heat semigroup $(e^{t\Delta})_{t\geq 0}$ on a network. 
For all subgraphs $\G'\subset \G$ induced by a set of nodes $\V'$ 
the following assertions are equivalent.
\begin{enumerate}[a)]
\item The ideal $L^2(\G')$ is invariant under the action of the semigroup $(e^{t\Delta})_{t\geq 0}$.
\item If $\mathsf{v} \in \partial \G'$, then $ \deg(\mathsf{v}) = \infty$.
\end{enumerate}
\end{prop}
\begin{proof}
We preliminarily observe that the orthogonal projection $P$ onto $L^2(\G')$ acts on a function $\psi$ by
$$
P\psi(x)=
\begin{cases}
\psi(x), & x \in \G', \\
0, & x \in \G \setminus \G'.
\end{cases}
$$
To see that \emph{b)} implies \emph{a)}, we have to prove that the conditions in Theorem \ref{ortho} hold. 
The second condition is clear, since $P \psi$ and $(\id-P) \psi$ have disjoint support.
So, we only prove that $P V \subset V$, i.e., that the continuity in the nodes is preserved under the action of $P$. 

On all internal nodes of $\G'$, $P\psi$ is continuous since the projection acts as the identity in a full neighbourhood of the node,
and on all internal node of $\G \setminus \G'$ $P\psi$ is continuous since $P\psi\equiv0$ in a full neighbourhood of the node.

It remains to prove continuity in the nodes on the boundary of $\G'$. 
We arbitrarily choose a node $\mathsf{v} \in \partial \G'$ and consider the star centred in $\mathsf{v}$ $\Gamma(\mathsf{v})$.
On each $\mathsf{e} \in \Gamma(\mathsf{v}) \cap \G\setminus\G'$, $\psi_{\mathsf{e}}\equiv 0$ and so $\psi_{\mathsf{e}}(\mathsf{v})=0$.
On the other side,$\psi(\mathsf{v})=0$ since $\deg(\mathsf{v})=\infty$, and so for each $\mathsf{e} \in \Gamma(\mathsf{v}) \cap \G'$ $\psi_{\mathsf{e}}(\mathsf{v})=0$.
As a consequence, defining 
$$
d^{P\psi}=
\begin{cases}
d^\psi_{\mathsf{v}} & \mathsf{v} \in V',\\
0 & \mbox{otherwise},
\end{cases}
$$
proves the continuity of $P\psi$.

To prove the converse implication, observe that the boundary space $\partial V \subset \ell^2(V)$ satisfies
\begin{equation}\label{eq:partialV}
\partial V:=\{ d^\psi: \psi \in V\} \subset \{(x_{\mathsf{v}})_{\mathsf{v} \in \V} \in \ell^2(\V): [\deg(\mathsf{v}')=  \infty \Rightarrow x_{\mathsf{v}'}=0]\}.
\end{equation}
Assume that $L^2(\G')$ is invariant. By Theorem \ref{ortho} $P\psi$ is continuous in the nodes whenever $\psi$ is continuous in the nodes.
In particular, $P \psi$ has to be continuous in all nodes $\mathsf{v} \in \G\setminus\G'$ that are adjacent to $\G'$
and for these nodes $P\psi(\mathsf{v})=0$. So, we arbitrarily choose a neighbourhood $N$ of $\mathsf{v}$ and $\psi \in V$.
On each point $x \in (N \cap \G') \setminus \{\mathsf{v}\}=:N'$ the projection $P$ acts as the identity, i.e $P\psi(x)=\psi(x)$. 
Further, the ideal $I$ is invariant and so $P\psi$ is a continuous function. We compute
$$
0=\lim_{N' \ni x \to \mathsf{v}}P\psi(x)= \lim_{N' \ni x \to \mathsf{v}} \psi(x)= \psi(\mathsf{v}).
$$
Since the choice of $\psi$ is arbitrary, $\deg(\mathsf{v})=\infty$ follows from Lemma~\ref{hilflemma3}.
\end{proof}

Proposition \ref{infiniteinvariance} allows us to characterize the irreducibility of $(e^{t\Delta})_{t\geq0}$
in terms of the connectedness of the graph $\mathsf G$. 
For finite graphs, irreducibility is known to be equivalent to the graph $\G$ being connected by paths.
Before proving similar results for infinite graphs, 
we prove the easy result that pathwise connectedness is equivalent to the topological one.
The result is probably known, but we were not able to find a reference in the literature.

\begin{prop}\label{connected}
The following assertions are equivalent.
\begin{enumerate}[a)]
\item The graph $\G$ is pathwise connected: for every two nodes $\mathsf{v}_1,\mathsf{v}_2 \in \V$ there exists a finite path connecting $\mathsf{v}_1$ and $\mathsf{v}_2$.
\item The graph $\G$ is topologically connected, i.e., if $\emptyset \neq \V_1, \V_2 \subset \V$ are sets such that
$$
\V_1 \cap \V_2 = \emptyset, \quad \mbox{and} \quad  \V_1 \cup \V_2 =\V,
$$
then there exists $ \mathsf{e} \in \E$ such that $\mathsf{e} \sim \V_1, \mathsf{e} \sim \V_2$.
\end{enumerate}
\end{prop}

\begin{proof}
We assume that \emph{a)} holds and fix a decomposition $\V=\V_1 \cup \V_2$. 
Since the graph is pathwise connected, for every $\mathsf{v}_1 \in \V_1, \mathsf{v}_2 \in \V_2$ there exists a path $P=[\mathsf{e}_1,\ldots,\mathsf{e}_\ell]$ of finite length $\ell$ connecting $\mathsf{v}_1$ to $\mathsf{v}_2$. 
The index
$$
i_0:=\max_{i=1,\ldots,\ell} \{i :\mathsf{e}_i \in \E_1\}
$$
an edge $\mathsf{e}_{i_0}$ that is adjacent to to both $\E_1$ and $\E_2$.

Conversely, assume that \emph{b)} holds. Fix two nodes $\mathsf{v}_1,\mathsf{v}_2 \in \V$ and define
$$
\V_1:= \bigcup_{n=1}^\infty \{\mathsf{v}' \in \V: d(\mathsf{v}_1,\mathsf{v}')=n\},\qquad \V_2:= \bigcup_{n=1}^\infty \{\mathsf{v}' \in \V: d(\mathsf{v}_2,\mathsf{v}')=n\}.
$$
If $\V_1 =\V_2$ there is nothing to prove.
If $\V_1 \cap \V_2=\emptyset$ then there exists by assumption an edge connecting both vertex sets.
This concludes the proof.
\end{proof}

In order to characterize irreducibility, we define the \emph{finite span} $\Sfin(\mathsf{e}) \subset \E$ 
of the edge $\mathsf{e}$ as the subgraph induced by the set of edges
$$
\mathsf{E}_{\mathrm{fin}}:=\{\mathsf{e}' \in \E: \mbox{there is a path from $\mathsf{e}$ to $\mathsf{e}'$ containing no infinite stars}\}.
$$
We say that the paths that have the above property have \emph{finite weight}.
Paths with \emph{infinite weight} are defined analogously.
We are now in the position of stating the main theorem of this Section.
\begin{thm}\label{infiniteirreducibility}
For a countable graph $\mathsf G$ the following assertions are equivalent.
\begin{enumerate}[a)]
\item The semigroup $(e^{t\Delta})_{t \geq 0}$ is irreducible. 
\item $\Sfin(\mathsf{e})= \G$ for one $\mathsf{e} \in \E$.
\item $\Sfin(\mathsf{e})= \G$ for all $\mathsf{e} \in \E$.
\end{enumerate}
\end{thm}
\begin{cor}
If $\G$ is a connected, locally finite network, then $(e^{t\Delta})_{t\geq 0}$ is irreducible.
\end{cor}
\begin{cor}\label{cor:numbersefin}
The number of non trivial, minimal invariant ideals of $L^2(\mathsf G)$ is the number of the different $\Sfin(\mathsf{e})$ contained in the graph.
\end{cor}

We split the proof of Theorem~\ref{infiniteirreducibility} in several steps.
The idea is to prove that the invariant ideals of the semigroup $(e^{t\Delta})_{t\geq 0}$ are of the form 
$\cup_i \Sfin(\mathsf{e}_i)$ for some family $\{\mathsf{e}_i\} \subset \E$.
As a preliminary remark, we observe that the only possible invariant ideals are of the form $L^2(\G')$, 
where $\G'$ is some subgraph of $\G$ induced by a subset of the node set.
To see this, recall that all ideals of $L^2(\G)$ have the form $L^2(\omega)$, 
where $\omega=\bigoplus_{\mathsf{e} \in \E} \omega_{\mathsf{e}} \subset \bigoplus_{\mathsf{e} \in \E}[0,1]$. 
Thus, we are claiming that if $L^2(\omega)$ is invariant, then
$|\omega_{\mathsf{e}}|\in\{0,1\}$, 
but this is a consequence of the irreducibility of the heat semigroup on $L^2[0,1]$.
We now show that ideals of the form $\Sfin(\mathsf{e})$ are invariant.

\begin{lemma}\label{hilflemma2}
Consider a connected graph $\G$ and $\mathsf{e} \in \E$. Then $\Sfin(\mathsf{e})$ is invariant under the action of $(e^{t\Delta})_{t\geq 0}$.
\end{lemma}
\begin{proof}
We use Theorem \ref{ortho}. To prove that both conditions hold, 
we arbitrarily choose $\mathsf{e} \in \E$ and denote by $P$ the projection onto $\Sfin(\mathsf{e})$. 
Observe that $P\psi(x)=\mathsf{1}_{\Sfin(\mathsf{e})}(x) \psi(x)$ for all $x \in \G$.
So, the first condition of Theorem \ref{ortho} holds since $P\psi$ and $(I-P)\psi$ have disjoint support. 

\smallskip
We prove that $P V\subset V$. 
Recall that the boundary of $\Sfin(\mathsf{e})$ consists of those nodes that are adjacent to $\Sfin(\mathsf{e})$ and to its complement.
So, one only has to prove continuity in the nodes $\mathsf{v}\in \partial \Sfin(\mathsf{e})$, and indeed
it suffices to show that all the nodes on the boundary of $\Sfin(\mathsf{e})$ have infinite degree. 
But this is clear as for, if $\mathsf{v}$ is on the boundary of $\Sfin(\mathsf{e})$ and has finite degree, all edges incident onto $\mathsf{v}$ are in $\Sfin(\mathsf{e})$ by definition, 
and so $\mathsf{v}$ is internal to $\Sfin(\mathsf{e})$.
\end{proof}

The next step is to identify the subgraphs of the form $\Sfin(\mathsf{e})$.

\begin{lemma}\label{hilflemma1}
Consider a connected graph $\G$ and a connected subgraph $\G'$. Consider the following assertions.
\begin{enumerate}
\item $\deg(\mathsf{v})=\infty$ for all nodes in $\partial \G'$.
\item There exists a path with finite weight between every $\mathsf{e},\mathsf{e}' \in \G'$.
\item The subgraph $\G'$ is the finite span $\Sfin(\mathsf{e}')$ of each of its edges.
\end{enumerate}
Then
\begin{enumerate}[a)]
\item $[1. \wedge 2.] \Leftrightarrow [3.]$ and
\item $[1.] \Rightarrow [\forall \mathsf{e} \in \G':\Sfin(\mathsf{e} ) \subset \G' ]$.
\end{enumerate}
\end{lemma}

\begin{proof}

To prove the first direction of $a)$ we fix a subgraph with the required properties.
We first observe that \emph{1.}\ implies $\G' \subset \Sfin(\mathsf{e}) $ for all
$\mathsf{e} \in \G'$.
Assume now that $\exists \mathsf{e}' \in \Sfin(\mathsf{e})\setminus \G'$. 
Without loss of generality, let $\mathsf{e}' \sim \G'$, i.e., $\mathsf{e}'\sim \mathsf{v}, \mathsf{v} \in \G'$ and assume that the boundary $\partial \G'$ only consists of $\mathsf{v}$. 
By hypothesis $\deg(\mathsf{v})= \infty$ and so there is no path with finite weight between $\mathsf{e}$ and $\mathsf{e}'$, which is a contradiction.

Conversely, if $\mathsf G'$ is the finite span of each of its edges then 
\emph{2.}\ is trivially true 
and \emph{1.}\ follows from the fact that $\mathsf G'$ is a finite span.
For if one node on the boundary would have finite degree, then all adjacent nodes would belong to the same finite span and hence to $\mathsf G'$. But this means that the node is internal to $\mathsf G'$, hence it does not belong to the boundary.

\emph{b)} We arbitrarily choose $\mathsf{e} \in \G'$, $\mathsf{e}' \not\in \G'$, and a path $P$ between $\mathsf{e}$ and $\mathsf{e}'$. 
By definition of $\partial \G'$, there exists $\mathsf{v} \in P \cap \partial \G'$. 
As a consequence $P$ has infinite weight and the proof is complete.
\end{proof}
The following is a straightforward consequence of the above lemma.
\begin{prop}\label{decompositiongraph}
Consider a connected graph $\G$. Then there exists $\E' \subset \E$ such that
$$
\bigcup_{\mathsf{e} \in \E'} \Sfin(\mathsf{e})= \G,
$$
and
$$ 
\partial \Sfin(\mathsf{e}) = \Sfin(\mathsf{e})\cap \Sfin(\mathsf{e}')=\partial \Sfin(\mathsf{e}'), \qquad \mathsf{e},\mathsf{e}' \in \E'.
$$
\end{prop}
Now we can characterise irreducibility.
\begin{proof}[Proof of Theorem~\ref{infiniteirreducibility}]
Observe that, as a consequence of Lemma \ref{hilflemma3}, $\psi(\mathsf{v})=0$ for all $\psi \in \V$ if and only if $\deg(\mathsf{v})=\infty$.
Also recall that irreducibility is equivalent to the fact that
the only invariant ideal is $L^2(\mathsf G)$.
In order to see that b) and c) are equivalent observe that if $\mathsf{e}' \in \Sfin(\mathsf{e})$, then $\Sfin(\mathsf{e})=\Sfin(\mathsf{e}')$.

Assume that a) holds. Then the only invariant ideal is $L^2(\mathsf G')$. 
Since by Lemma~\ref{hilflemma2} $L^2(\Sfin(\mathsf{e}))$ is invariant for all $\mathsf{e} \in \E$, c) follows.
Conversely, assume that c) holds and that $L^2(\mathsf G')$ is invariant. We observe that 
the projection $P\psi$ of a function $\psi$ on $L^2(\G')$ vanishes in all points of $\G\setminus\G'$ and so, 
by continuity it vanishes in all points of the boundary of $\G'$.
Since on $L^2(\mathsf G')$ $P$ coincides with the identity,
we deduce that each function $\psi \in V$ also has to vanish on all points of the boundary of $\G'$
and so we conclude by Lemma \ref{hilflemma3}, that all those points $\mathsf{v}'$ satisfy $\deg(\mathsf{v}')=\infty.$
So, for all $\mathsf{e} \in \mathsf G'$, $\Sfin(\mathsf{e}) \subset \mathsf G'$, hence $\mathsf G' = \mathsf G$
and the proof is complete.
\end{proof}
This theorem helps to establish a relation to the well-known result connecting the eigenvalues
of the Laplace matrix of a graph to the number of connected components.
We start by a definition.
\begin{defi}
A graph is \emph{$\Delta$-connected}, 
if the heat equation on the corresponding network is irreducible in $L^2(\mathsf G)$.
The number of $\Delta$-connected components is the number of non trivial, 
minimal invariant ideals of the corresponding heat equation.
\end{defi}
For finite graphs $\Delta$-connectedness and topological connectedness, as well as number 
of invariant ideals of the Laplacian and multiplicity of the eigenvalue 0 coincide
and this is reflected in a well-known elementary theorem from basic graph theory.
\begin{thm}[Connected components and multiplicity of $\lambda_0$]\label{theo:connected}
For a finite graph, the number of connected components of a graph $G$ is the multiplicity of the eigenvalue 0 of the Laplace matrix of the graph.
\end{thm}
For the $\ell^2$-Laplacian matrix on non finite networks,
the value 0 does not need to be an eigenvalue since the constant vector $\mathsf 1$
is not part of $\ell^2(\mathsf V)$, see~\cite{Bel09} for a detailed discussion
of spectral properties of the adjacency matrix.

However, Corollary~\ref{cor:numbersefin} shows that the theorem carries over to the new situation if topological connectedness is replaced by $\Delta$-connectedness and the multiplicity of the eigenvalues is replaced by the number of invariant ideals. So, we reformulate Corollary~\ref{cor:numbersefin} in the following form.
\begin{thm}[Connected components and invariant ideals]\label{theo:connectedagain}
For a countable graph,
the number of $\Delta$-connected components of a graph $\mathsf G$
equals the number of maximal invariant ideals of heat equation on the corresponding network.
\end{thm}

\section{Acknowledgements}
A major part of this work was accomplished as the author was
in the ``Graduate School for Mathematical Analysis of Evolution, Information and Complexity'' of the University of Ulm.
Delio Mugnolo and Robin Nittka are gratefully acknowledged 
for discussion and valuable help.

\bibliographystyle{unsrt} 
\bibliography{/home/cardanobile/Dropbox/literatur}
\end{document}